\documentclass{amsart}

\usepackage[all]{xy}
\usepackage{amssymb}

\theoremstyle{plain}
\newtheorem{theorem}{Theorem}[section]
\newtheorem{proposition}[theorem]{Proposition}
\newtheorem{lemma}[theorem]{Lemma}
\newtheorem{corollary}[theorem]{Corollary}

\theoremstyle{definition}

\newtheorem{definition}[theorem]{Definition}

\newtheorem{remark}[theorem]{Remark}

\newcommand{\calA}{\mathcal{A}}

\newcommand{\calE}{\mathcal{E}}
\newcommand{\calF}{\mathcal{F}}

\newcommand{\calP}{\mathcal{P}}
\newcommand{\calQ}{\mathcal{Q}}

\newcommand{\calT}{\mathcal{T}}
\newcommand{\calV}{\mathcal{V}}
\newcommand{\calW}{\mathcal{W}}

\newcommand{\boldC}{\mathbf{C}}
\newcommand{\boldD}{\mathbf{D}}

\newcommand{\Cb}{\mathbf{C}^{\mathrm{b}}}
\newcommand{\Cperf}{\mathbf{C}^{\mathrm{perf}}}
\newcommand{\Db}{\mathbf{D}^{\mathrm{b}}}
\newcommand{\Dperf}{\mathbf{D}^{\mathrm{perf}}}
\newcommand{\Kb}{\mathbf{K}^{\mathrm{b}}}

\newcommand{\Aut}{\mathrm{Aut}}
\newcommand{\cone}{\mathrm{cone}}
\newcommand{\ev}{\mathrm{ev}}

\newcommand{\Hom}{\mathrm{Hom}}
\newcommand{\id}{\mathrm{id}}
\newcommand{\iso}{\mathrm{iso}}
\newcommand{\modfg}{\textrm{-}\mathrm{mod}}

\newcommand{\objects}{\mathrm{obj}}
\newcommand{\od}{\mathrm{od}}

\newcommand{\proj}{\textrm{-}\mathrm{proj}}
\newcommand{\qis}{\mathrm{qis}}

\newcommand{\tri}{\mathrm{tri}}

\begin{document}

\title[Determinants of perfect complexes and Euler
characteristics]{Determinants of perfect complexes and Euler
characteristics in relative $K_0$-groups}
\author{Manuel Breuning}
\address{Department of Mathematics, King's College London, Strand,
London WC2R 2LS, United Kingdom} \email{manuel.breuning@kcl.ac.uk}
\date{December 8, 2008}

\begin{abstract}
We study the $K_0$ and $K_1$-groups of exact and triangulated
categories of perfect complexes, and we apply the results to show
how determinant functors on triangulated categories can be used for
the construction of Euler characteristics in relative algebraic
$K_0$-groups.
\end{abstract}

\maketitle

\section{Introduction}

Let $R$ be a ring and $P$ a cochain complex of finitely generated
projective $R$-modules. Then the Euler characteristic of $P$ is
defined to be the element $\chi(P)=\sum_{i} (-1)^i[P^i]$ in the
Grothendieck group $K_0(R)$. The definition of this Euler
characteristic can easily be generalized to perfect complexes of
$R$-modules, i.e.\ complexes which in the derived category become
isomorphic to bounded complexes of finitely generated projective
$R$-modules.

Now suppose that we have a ring homomorphism $R\to S$. Recall that then
there exists an exact sequence of $K$-groups
\begin{equation*}
K_1(R) \longrightarrow K_1(S) \longrightarrow K_0(R,S)
\longrightarrow K_0(R) \longrightarrow K_0(S),
\end{equation*}
where $K_0(R,S)$ is Swan's relative $K_0$-group. Let $P$ be a
perfect complex of $R$-modules for which $\chi(P)$ lies in the
kernel of $K_0(R)\to K_0(S)$. In this situation, can we find a
canonical preimage of $\chi(P)$ in $K_0(R,S)$? In certain cases this
is indeed possible, provided the complex $S\otimes_R P$ is endowed
with some natural additional data, and in recent years such refined
Euler characteristics in $K_0(R,S)$ have found applications in
arithmetic algebraic geometry. Sometimes they can be constructed
using only elementary methods (see \cite{Burns04}), however a more
general and conceptual approach is the use of determinant functors
on exact categories of complexes (see \cite{BurnsFlach01},
\cite{BreuningBurns05}).

In the applications, the relevant complexes often lie in derived
categories and not in exact categories. Therefore it would be more
natural to use determinant functors on triangulated categories (as
defined in \cite{Breuning06}) for the construction of Euler
characteristics in $K_0(R,S)$. The purpose of this paper is to show
that this is indeed possible.

The crucial observation is that the $K$-groups of a triangulated
category of perfect complexes are closely related to classical
$K$-groups. To state the precise result, we let $R$ be a ring and
denote the triangulated category of perfect complexes of $R$-modules
by $\Dperf(R)$. In \cite{Breuning06} the $K$-groups
$K_i\big(\Dperf(R)\big)$ for $i=0$ and $1$ were defined in terms of
a universal determinant functor on $\Dperf(R)$. These groups are
related to the usual $K$-groups of $R$ as follows.

\begin{theorem}
\label{theorem_intro} There exist canonical homomorphisms
\begin{equation*}
K_i(R)\longrightarrow K_i\big(\Dperf(R)\big)
\end{equation*}
for $i=0$ and $1$. It is an isomorphism for $i=0$ and surjective for
$i=1$. If $R$ is regular then it is an isomorphism for $i=1$.
\end{theorem}

The proof of this result involves studying the relation between
determinant functors on the exact category of finitely generated
projective $R$-modules, the exact category of perfect complexes of
$R$-modules, and the triangulated category of perfect complexes of
$R$-modules. Using this theorem it is then not difficult to show
that in the construction of Euler characteristics in a relative
$K_0$-group the determinant functors on exact categories can be
replaced by determinant functors on triangulated categories.

This paper is structured as follows. We study determinant functors
on exact and triangulated categories of perfect complexes in \S
\ref{section_perfect_exact} and \S
\ref{section_perfect_triangulated} respectively. In \S
\ref{section_homotopy_fibre} we
recall the homotopy fibre of a monoidal functor of Picard categories.
Finally in \S \ref{section_Euler_triangulated} we
explain the construction of Euler characteristics
using determinant functors on triangulated categories.

In \S \ref{section_example} we will give an example of a non-regular
ring $R$ for which the canonical surjection $K_1(R)\to
K_1\big(\Dperf(R)\big)$ from Theorem \ref{theorem_intro} is not
injective. Furthermore using the same example we will show that the
isomorphism $K_1(\calA)\cong K_1\big(\Db(\calA)\big)$ which was
proved in \cite[\S 5]{Breuning06} for abelian categories $\calA$
does not generalize to exact categories. The results from \S
\ref{section_example} are not used in the subsequent sections.

We refer the reader to \cite{Breuning06} for the definition and
properties of determinant functors on exact and triangulated
categories. Moreover we use the same notations as in
\cite{Breuning06}, in particular we recall that if $\calP$ is a
Picard category, $\calE$ is an exact category and $w$ a class of
morphism in $\calE$ which contains all isomorphisms and is closed
under composition then $\det((\calE,w),\calP)$ denotes the category
whose objects are the determinant functors $f=(f_1,f_2):
(\calE,w)\to\calP$, and we omit $w$ from the notation if $w=\iso$ is
the class of all isomorphisms.

\section{Determinant functors on exact categories of perfect
complexes} \label{section_perfect_exact}

Let $R$ be a ring (associative with unit). Unless otherwise stated
all $R$-modules will be left $R$-modules, and all complexes of
$R$-modules will be cochain complexes of left $R$-modules. The
abelian category of all complexes of $R$-modules will be denoted by
$\boldC(R)$. If $M$ is an $R$-module then $M[0]$ denotes the complex
which consists of the module $M$ in degree $0$ and of zero modules
in all other degrees.

Let $\calE$ be a full exact subcategory of $\boldC(R)$. A morphism
in $\calE$ will be called a quasi-isomorphism if it becomes a
quasi-isomorphism after embedding in $\boldC(R)$, i.e.\ if it
induces an isomorphism on cohomology. The class of all
quasi-isomorphisms in $\calE$ will be denoted by $\qis$. In the
following we will frequently use the following lemma (which is
proved in the same way as \cite[Corollary 2.12]{Knudsen02}).

\begin{lemma}
Assume that $\calE$ is a full exact subcategory of $\boldC(R)$ which
is closed under mapping cones (i.e.\ if $a$ is any morphism in
$\calE$ then $\cone(a)$ is an object of $\calE$). Let $f=(f_1,f_2):
(\calE,\qis)\to\calP$ be a determinant functor. Then for any
homotopic quasi-isomorphisms $a$ and $b$ in $\calE$ we have
$f_1(a)=f_1(b)$ in $\calP$.
\end{lemma}

We remark that the assumptions of the lemma are satisfied for the
subcategories $\Cb(R\proj)$, $\Cb(R\modfg)$ and $\Cperf(R)$ which
are considered later in this paper.

Let $R$-proj denote the category of finitely generated projective
$R$-modules and $\Cb(R\proj)$ the category of bounded complexes of
objects of $R\proj$. The categories $R\proj$ and $\Cb(R\proj)$ are
both exact and the embedding $R\proj\to\Cb(R\proj)$, $M\mapsto
M[0]$, is an exact functor. For every Picard category $\calP$ we
obtain an induced functor
\begin{equation}
\label{equation_proj_exact} \det\big((\Cb(R\proj),\qis),\calP\big)
\longrightarrow \det(R\proj,\calP)
\end{equation}
which by \cite[Theorem 2.3]{Knudsen02} is an equivalence of
categories.

\begin{definition}
A complex $A$ of $R$-modules is said to be \textit{perfect} if there
exists a bounded complex of finitely generated projective
$R$-modules $P$ and a quasi-isomorphism $P\to A$.
\end{definition}

The full subcategory of $\boldC(R)$ consisting of all perfect
complexes will be denoted by $\Cperf(R)$. Note that $\Cperf(R)$ is
an exact category and that the embedding $\Cb(R\proj)\to\Cperf(R)$
is an exact functor. Hence for every Picard category $\calP$ we
obtain an induced functor
\begin{equation}
\label{equation_perfect_exact} \det\big((\Cperf(R),\qis),\calP\big)
\longrightarrow \det\big((\Cb(R\proj),\qis),\calP\big).
\end{equation}

\begin{lemma}
\label{lemma_perfect_exact} The functor
(\ref{equation_perfect_exact}) is an equivalence of categories.
\end{lemma}

\begin{remark}
Before sketching the proof of this lemma we recall a useful property
of objects in $\Cb(R\proj)$ which will be used repeatedly in this
paper. Let $s: A\to B$ be a quasi-isomorphism in $\boldC(R)$ and let
$b: P\to B$ be any morphism of complexes where $P$ is an object in
$\Cb(R\proj)$ (or, more generally, $P$ is a bounded above complex of
projective $R$-modules). Then there exists a morphism of complexes
$a: P\to A$ such that the diagram
\begin{equation*}
\xymatrix{ & A \ar[d]^s \\ P \ar@{-->}[ur]^a \ar[r]^b & B }
\end{equation*}
commutes up to homotopy. Furthermore this morphism $a$ is unique up
to homotopy.
\end{remark}

\begin{proof}[Proof of Lemma \ref{lemma_perfect_exact}]
This is essentially known (see e.g.\ \cite[Theorem
2]{KnusemMumford76}), so we only sketch the main idea. For every
perfect complex $C$ we fix a quasi-isomorphism $q_C: P_C\to C$ where
$P_C$ is in $\Cb(R\proj)$. Then we can construct a functor
\begin{equation}
\label{equation_exact_inverse}
\det\big((\Cb(R\proj),\qis),\calP\big) \longrightarrow
\det\big((\Cperf(R),\qis),\calP\big)
\end{equation}
as follows. Given a determinant functor $f=(f_1,f_2):
(\Cb(R\proj),\qis) \to \calP$ we define $g_1:
\Cperf(R)_{\qis}\to\calP$ by $g_1(C):=f_1(P_C)$ and for a
quasi-isomorphism $a: C\to D$ by $g_1(a):=f_1(b)$ where $b: P_C\to
P_D$ is a quasi-isomorphism such that $q_D\circ b$ is homotopic to
$a\circ q_C$ (this is well-defined because such a map $b$ is unique
up to homotopy). If $\Delta: 0\to A\to B\to C\to 0$ is a short exact
sequence in $\Cperf(R)$ then there exists a short exact sequence
$\Delta': 0\to A'\to B'\to C'\to 0$ in $\Cb(R\proj)$ and a
commutative diagram
\begin{equation*}
\xymatrix{ 0 \ar[r] & A' \ar[r] \ar[d]^a & B' \ar[r] \ar[d]^b & C'
\ar[r] \ar[d]^c & 0 \\
0 \ar[r] & A \ar[r] & B \ar[r] & C \ar[r] & 0 }
\end{equation*}
where the vertical maps are quasi-isomorphisms. Furthermore there
are quasi-isomorphisms $u: P_A\to A'$, $v: P_B\to B'$ and $w: P_C\to
C'$ such that $a\circ u$, $b\circ v$ and $c\circ w$ are homotopic to
$q_A$, $q_B$ and $q_C$ respectively. We define $g_2(\Delta):
g_1(B)\to g_1(A)\otimes g_1(C)$ to be $g_2(\Delta):=(f_1(u)\otimes
f_1(w))^{-1}\circ f_2(\Delta')\circ f_1(v)$. One can verify that
$g_2(\Delta)$ is well-defined and that $(g_1,g_2)$ is a determinant
functor in $\det\big((\Cperf(R),\qis),\calP\big)$.

It is obvious how to define the functor
(\ref{equation_exact_inverse}) on morphisms of determinant functors.
One then easily checks that this functor
(\ref{equation_exact_inverse}) is a quasi-inverse of the functor
(\ref{equation_perfect_exact}).
\end{proof}

\begin{corollary}
\label{corollary_perfect_exact} For every Picard category $\calP$
the canonical functor
\begin{equation*}
\det\big((\Cperf(R),\qis),\calP\big) \to \det(R\proj,\calP)
\end{equation*}
which is induced by the embedding $R\proj\to\Cperf(R)$ is an
equivalence of categories. Hence the canonical map
\begin{equation*}
K_i(R)\longrightarrow K_i\big(\Cperf(R),\qis\big)
\end{equation*}
is an isomorphism.
\end{corollary}

\begin{proof}
The first statement is immediate by composing the equivalences
(\ref{equation_proj_exact}) and (\ref{equation_perfect_exact}). The
second statement follows from the first statement and the relation
between $K$-groups and universal determinant functors, see Lemma
\ref{lemma_equivalence} below.
\end{proof}

\begin{lemma}
\label{lemma_equivalence} Let $\calE$ and $\calE'$ be exact
category. Let $w$ and $w'$ be classes of morphisms in $\calE$ and
$\calE'$ respectively which contain all isomorphisms and are closed
under composition. Let $F:\calE\to\calE'$ be an exact functor with
$F(w)\subseteq w'$. Then for every Picard category $\calP$ there is
an induced functor
\begin{equation}
\label{equation_principle_1}
\det((\calE',w'),\calP)\longrightarrow\det((\calE,w),\calP),
\end{equation}
and there are induced homomorphisms
\begin{equation}
\label{equation_principle_2}
K_i(\calE,w)\longrightarrow K_i(\calE',w')
\end{equation}
for $i=0,1$. The following are equivalent.
\begin{itemize}
\item[(i)] For every Picard category $\calP$ the
functor (\ref{equation_principle_1}) is an
equivalence of categories.
\item[(ii)] The homomorphisms (\ref{equation_principle_2}) are
isomorphisms for $i=0$ and $1$.
\end{itemize}
\end{lemma}

\begin{proof}
Let $f: (\calE,w)\to\calV$ and $f': (\calE',w')\to\calV'$ be
universal determinant functors, and let $M: \calV\to\calV'$ be a
monoidal functor such that $M\circ f\cong f'\circ F$ (by the
definition of a universal determinant functor such an $M$ exists and
is unique up to isomorphism). Then there exists the following
diagram of categories and functors
\begin{equation*}
\xymatrix@C+1cm{
\det((\calE',w'),\calP) \ar[r]^{g\mapsto g\circ F} & \det((\calE,w),\calP) \\
\Hom^\otimes(\calV',\calP) \ar[u] \ar[r]^{N\mapsto N\circ M} &
\Hom^\otimes(\calV,\calP) \ar[u] }
\end{equation*}
which commutes up to natural isomorphism. Since the vertical
functors are equivalences, it follows that the top horizontal
functor is an equivalence if and only if the bottom horizontal
functor is an equivalence. Clearly the top horizontal functor is the
functor (\ref{equation_principle_1}). On the other hand it is easy
to see that the bottom horizontal functor is an equivalence for all
$\calP$ if and only if the monoidal functor $M:\calV\to\calV'$ is an
equivalence of Picard categories, and by \cite[Lemma
2.2]{Breuning06} this is the case if and only if
$K_i(\calE,w)=\pi_i(\calV)\xrightarrow{\pi_i(M)}\pi_i(\calV')=K_i(\calE',w')$
is an isomorphism for $i=0$ and $i=1$.
\end{proof}

We remark that a statement similar to Lemma \ref{lemma_equivalence}
is also valid for exact functors $F:\calT\to\calT'$ of triangulated
categories $\calT, \calT'$ and for certain functors
$F:\calE\to\calT$ from an exact category $\calE$ to a triangulated
category $\calT$.

\section{Determinant functors on triangulated categories of perfect
complexes} \label{section_perfect_triangulated}

Let $R$ be a ring. The derived category of the abelian category of
all $R$-modules will be denoted by $\boldD(R)$. Thus the objects in
$\boldD(R)$ are the complexes of $R$-modules, and a morphism $A\to
B$ in $\boldD(R)$ is an equivalence class of diagrams
$A\xleftarrow{s} C\xrightarrow{a} B$ where $a$ is any morphism of
complexes and $s$ is a quasi-isomorphism. The category $\boldD(R)$
is triangulated, where as in \cite[bottom of p.\
28]{BreuningBurns05} we choose the triangulation in which a triangle
is distinguished if it is isomorphic to a triangle of the form
$A\xrightarrow{a} B \to \cone(a)\to A[1]$, where $B\to\cone(a)$ is
the canonical inclusion and $\cone(a)\to A[1]$ is the negative of
the canonical projection. The full subcategory of $\boldD(R)$
consisting of all perfect complexes will be denoted by $\Dperf(R)$.
One easily verifies that $\Dperf(R)$ is a triangulated subcategory
of $\boldD(R)$. We let $I: \Cperf(R)\to\Dperf(R)$ denote the
canonical functor.

\begin{lemma}
\label{lemma_C_D_perf_functor} The functor $I:
\Cperf(R)\to\Dperf(R)$ induces a functor
\begin{equation*}
\begin{split}
\det\big(\Dperf(R),\calP\big) & \longrightarrow
\det\big((\Cperf(R),\qis),\calP\big), \\
g & \longmapsto g\circ I,
\end{split}
\end{equation*}
for every Picard category $\calP$, and it therefore induces
homomorphisms
\begin{equation*}
K_i\big(\Cperf(R),\qis\big)\longrightarrow K_i\big(\Dperf(R)\big)
\end{equation*}
for $i=0, 1$.
\end{lemma}

\begin{proof}
We first recall that for every short exact sequence $\Delta: 0\to
A\xrightarrow{a} B\xrightarrow{b} C\to 0$ of complexes of
$R$-modules there exists an associated distinguished triangle
$\widehat{\Delta}: A\xrightarrow{a} B\xrightarrow{b}
C\xrightarrow{c} TA$ in $\boldD(R)$ where $c$ is the morphism
$C\xleftarrow{s}\mathrm{cone}(a)\to TA$ (here $s:
\mathrm{cone}(a)\to C$ is the canonical quasi-isomorphism and
$\mathrm{cone}(a)\to TA$ is the negative of the canonical
projection).

Let $g=(g_1,g_2): \Dperf(R)\to\calP$ be a determinant functor. The
functor $I: \Cperf(R)\to\Dperf(R)$ sends quasi-isomorphisms in
$\Cperf(R)$ to isomorphisms in $\Dperf(R)$. We can therefore define
$f_1:=g_1\circ I: \Cperf(R)_{\qis}\to\calP$. For a short exact
sequence $\Delta: 0\to A\xrightarrow{a} B\xrightarrow{b} C\to 0$ in
$\Cperf(R)$ we define $f_2(\Delta):=g_2(\widehat{\Delta}): f_1(B)\to
f_1(A)\otimes f_1(C)$. It is not difficult to verify that
$(f_1,f_2)$ is a determinant functor on $(\Cperf(R),\qis)$ which we
will denote by $g\circ I$. For a morphism $\lambda: g\to g'$ of
determinant functors $g, g': \Dperf(R)\to\calP$ it is clear how to
define the morphism $\lambda\circ I: g\circ I\to g'\circ I$. We
obtain a functor
\begin{equation*}
\det\big(\Dperf(R),\calP\big) \longrightarrow
\det\big((\Cperf(R),\qis),\calP\big).
\end{equation*}

Now if $f: (\Cperf(R),\qis) \to \calV$ and $g: \Dperf(R) \to \calW$
are universal determinant functors then there exists a monoidal
functor $M: \calV \to \calW$ such that $M\circ f$ and $g\circ I$ are
isomorphic. Furthermore $M$ is unique up to isomorphism. Therefore
$M$ induces well-defined group homomorphisms
\begin{equation*}
K_i\big(\Cperf(R),\qis\big) = \pi_i(\calV) \xrightarrow{\pi_i(M)}
\pi_i(\calW) = K_i\big(\Dperf(R)\big)
\end{equation*}
for $i=0,1$.
\end{proof}

We will now study the maps $K_i\big(\Cperf(R),\qis\big)\to
K_i\big(\Dperf(R)\big)$ in more detail, first for arbitrary rings
(Proposition \ref{proposition_perf_universal}) and then for regular
rings (Proposition \ref{proposition_perf_regular}).

\begin{proposition}
\label{proposition_perf_universal} Let $R$ be an arbitrary ring.
Then the homomorphism
\begin{equation*}
K_i\big(\Cperf(R),\qis\big)\longrightarrow K_i\big(\Dperf(R)\big)
\end{equation*}
is bijective for $i=0$ and surjective for $i=1$.
\end{proposition}

\begin{proof}
Let $f=(f_1,f_2):(\Cperf(R),\qis)\to\calV$ be a universal
determinant functor. We will construct a universal determinant
functor for the triangulated category $\Dperf(R)$ by identifying
certain morphisms in the Picard category $\calV$.

We first note that the functor $f_1: \Cperf(R)_{\qis}\to\calV$
naturally induces a functor $\tilde{f}_1: \Dperf(R)_{\iso}\to\calV$.
Indeed, if $A$ is an object in $\Dperf(R)$ then we let
$\tilde{f}_1(A):=f_1(A)$, and if $a: A\to B$ is an isomorphism in
$\Dperf(R)$ then we let $\tilde{f}_1(a):=f_1(t)\circ f_1(s)^{-1}$
for any quasi-isomorphisms $s: C\to A, t: C\to B$ in $\Cperf(R)$
such that $a=t\circ s^{-1}$ in $\Dperf(R)$. Note that
$f_1=\tilde{f}_1\circ I$.

Now let $S' \subseteq
\bigcup_{(X,Y)}\Hom_{\calV}(X,Y)\times\Hom_{\calV}(X,Y)$ (where the
union is over all pairs $(X,Y)$ of objects of $\calV$) be the class
consisting of the following pairs of morphisms in $\calV$. If
$\Delta_i: 0\to A_i\to B_i\to C_i\to 0$ are short exact sequences in
$\Cperf(R)$ for $i=1,2$ and
\begin{equation*}
\xymatrix{ \widehat{\Delta_1}:{\,} A_1 \ar[r] \ar@<2.6ex>[d]^a & B_1
\ar[r] \ar[d]^b & C_1 \ar[r] \ar[d]^c & TA_1 \ar[d]^{Ta} \\
\widehat{\Delta_2}:{\,} A_2 \ar[r] & B_2 \ar[r] & C_2 \ar[r] & TA_2
}
\end{equation*}
is a commutative diagram in $\Dperf(R)$ with isomorphisms $a, b, c$
in $\Dperf(R)$, then the pair
\begin{equation*}
\big((\tilde{f}_1(a)\otimes \tilde{f}_1(c))\circ f_2(\Delta_1),\quad
f_2(\Delta_2)\circ \tilde{f}_1(b)\big)
\end{equation*}
belongs to $S'$.

Let $Q': \calV\to\calV/S'$ be the quotient Picard category (cf.\
Lemma \ref{lemma_quotient} below), and let $\overline{f}_1$ be the
composite functor $\Dperf(R)_{\iso} \xrightarrow{\tilde{f}_1} \calV
\xrightarrow{Q'} \calV/S'$. If $\Delta: A\to B\to C\to TA$ is a
distinguished triangle in $\Dperf(R)$ then there exists a short
exact sequence $\Delta_1: 0\to A_1\to B_1\to C_1\to 0$ in
$\Cperf(R)$ and isomorphisms $a,b,c$ in $\Dperf(R)$ such that
\begin{equation*}
\xymatrix{ \Delta:{\,} A \ar[r] \ar@<2.4ex>[d]^{a} & B \ar[r]
\ar[d]^{b} & C
\ar[r] \ar[d]^{c} & TA \ar[d]^{Ta} \\
\widehat{\Delta_1}:{\,} A_1 \ar[r] & B_1 \ar[r] & C_1 \ar[r] & TA_1
}
\end{equation*}
commutes in $\Dperf(R)$. We define $\overline{f}_2(\Delta):
\overline{f}_1(B)\to \overline{f}_1(A)\otimes \overline{f}_1(C)$ to
be $\overline{f}_2(\Delta):=(\overline{f}_1(a^{-1})\otimes
\overline{f}_1(c^{-1})) \circ f_2(\Delta_1) \circ
\overline{f}_1(b)$. Our definition of $S'$ guarantees that
$\overline{f}_2(\Delta)$ is well-defined.

Now $\overline{f}=(\overline{f}_1,\overline{f}_2):
\Dperf(R)\to\calV/S'$ does not necessarily satisfy the associativity
axiom. Let $S''$ be the set of pairs of morphisms in $\calV/S'$
which must be identified in order for the associativity axiom to
hold. More precisely, if
\begin{equation*}
\xymatrix{ A \ar[r] \ar@{=}[d] & B \ar[r] \ar[d] & C' \ar[r] \ar[d]
& TA \ar@{=}[d] \\
A \ar[r] & C \ar[r] \ar[d] & B' \ar[r] \ar[d] & TA \\
& A' \ar@{=}[r] \ar[d] & A' \ar[d] & \\
& TB \ar[r] & TC' & }
\end{equation*}
is an octahedral diagram in $\Dperf(R)$, then the pair
\begin{equation*}
\big(\varphi \circ (\id\otimes \overline{f}_2(\Delta_{\mathrm{v2}}))
\circ \overline{f}_2(\Delta_{\mathrm{h2}}),\quad
(\overline{f}_2(\Delta_{\mathrm{h1}})\otimes\id) \circ
\overline{f}_2(\Delta_{\mathrm{v1}})\big)
\end{equation*}
belongs to $S''$. Here $\Delta_{\mathrm{h1}}$ and
$\Delta_{\mathrm{h2}}$ (resp.\ $\Delta_{\mathrm{v1}}$ and
$\Delta_{\mathrm{v2}}$) denote the first and second horizontal
(resp.\ vertical) distinguished triangles in the octahedral diagram,
and $\varphi$ is the associativity constraint in the Picard category
$\calV/S'$.

Let $Q'': \calV/S\to (\calV/S')/S''=:\calW$ be the quotient Picard
category (cf.\ Lemma \ref{lemma_quotient}). Let $g_1:=Q''\circ
\overline{f}_1: \Dperf(R)_{\iso}\to\calW$ and for every
distinguished triangle $\Delta: A\to B\to C\to TA$ in $\Dperf(R)$
let $g_2(\Delta)$ be the map
$g_1(B)=Q''(\overline{f}_1(B))\xrightarrow{Q''(\overline{f}_2(\Delta))}
Q''(\overline{f}_1(A)\otimes \overline{f}_1(C))=g_1(A)\otimes
g_1(C)$. It is then easy to check that $g=(g_1,g_2):
\Dperf(R)\to\calW$ is a determinant functor. Furthermore $Q\circ
f=g\circ I$ where $Q:=Q''\circ Q': \calV\to\calW$.

We claim that $g: \Dperf(R)\to\calW$ is universal. For this we must
show that for every Picard category $\calP$ the functor
\begin{equation}
\label{equation_f_D_W}
\begin{split}
\Hom^\otimes(\calW,\calP) & \longrightarrow
\det\big(\Dperf(R),\calP\big),
\\
M & \longmapsto M\circ g,
\end{split}
\end{equation}
is an equivalence of categories.

Let $h:\Dperf(R)\to\calP$ be any determinant functor. Then there
exists a monoidal functor $N: \calV\to \calP$ such that the
determinant functors $N\circ f$ and $h\circ I$ are isomorphic. One
easily sees that $N$ factors as
$\calV\xrightarrow{Q}\calW\xrightarrow{\overline{N}}\calP$ for a
unique monoidal functor $\overline{N}: \calW\to\calP$. Thus the
determinant functors $\overline{N}\circ g\circ I$ and $h\circ I$ are
isomorphic in $\det\big((\Cperf(R),\qis),\calP\big)$, and this
implies that the determinant functors $\overline{N}\circ g$ and $h$
are isomorphic in $\det(\Dperf(R),\calP)$. Hence the functor
(\ref{equation_f_D_W}) is essentially surjective.

To show that (\ref{equation_f_D_W}) is fully faithful, we consider
the following commutative diagram.
\begin{equation*}
\xymatrix@C+1.5cm{ \Hom^\otimes(\calW,\calP) \ar[r]^{M\mapsto M\circ
g} \ar[d]^{M\mapsto M\circ Q} & \det\big(\Dperf(R),\calP\big)
\ar[d]^{h\mapsto h\circ I} \\
\Hom^\otimes(\calV,\calP) \ar[r]^{N\mapsto N\circ f} &
\det\big((\Cperf(R),\qis),\calP\big) }
\end{equation*}
The bottom horizontal functor is fully faithful since $f:
(\Cperf(R),\qis)\to\calV$ is a universal determinant functor.
Furthermore it is easy to see that the left vertical functor is
fully faithful and that the right vertical functor is faithful. It
follows that the top horizontal functor is fully faithful as
required. This finishes the proof of the claim that $g$ is
universal.

By Lemma \ref{lemma_quotient} the homomorphisms $\pi_i(Q')$ and
$\pi_i(Q'')$ are bijective for $i=0$ and surjective for $i=1$.
Therefore the same is true for $\pi_i(Q)=\pi_i(Q'')\circ\pi_i(Q')$.
It follows that the homomorphism
\begin{equation*}
K_i\big(\Cperf(R),\qis\big)=\pi_i(\calV)\xrightarrow{\pi_i(Q)}
\pi_i(\calW)=K_i\big(\Dperf(R)\big)
\end{equation*}
is bijective for $i=0$ and surjective for $i=1$.
\end{proof}

The following lemma describes the quotient Picard category which was
used in the proof of Proposition \ref{proposition_perf_universal}.

\begin{lemma}
\label{lemma_quotient} Let $\calV$ be a Picard category and assume
that we are given a class $S\subseteq \bigcup_{(X,Y)}
\big(\Hom_\calV(X,Y)\times\Hom_\calV(X,Y)\big)$ where the union is
over all pairs $(X,Y)$ of objects in $\calV$. Then there exists a
Picard category $\calV/S$ and a monoidal functor $Q:
\calV\to\calV/S$ such that the following two properties are
satisfied.
\begin{enumerate}
\item $Q(\alpha)=Q(\alpha')$ whenever $(\alpha,\alpha')\in S$
\item If $\calP$ is any Picard category and $M: \calV\to\calP$ a monoidal
functor such that $M(\alpha)=M(\alpha')$ in $\calP$ whenever
$(\alpha,\alpha')\in S$ then there exists a unique monoidal functor
$N:\calV/S \to\calP$ making the diagram
\begin{equation*}
\xymatrix{
\calV \ar[r]^{Q} \ar[d]_{M} & \calV/S \ar@{-->}[dl]^{N} \\
\calP }
\end{equation*}
commutative.
\end{enumerate}

If $\calV$ is small then the functor $Q: \calV\to\calV/S$ induces an
isomorphism $\pi_0(Q): \pi_0(\calV)\cong\pi_0(\calV/S)$ and a
surjection $\pi_1(Q): \pi_1(\calV)\to\pi_1(\calV/S)$. If $S$
contains a pair $(\alpha,\alpha')$ with $\alpha\neq\alpha'$ then the
surjection $\pi_1(Q): \pi_1(\calV)\to\pi_1(\calV/S)$ is not an
isomorphism.
\end{lemma}

\begin{proof}
Suppose that for every pair $(X,Y)$ of objects in $\calV$ we have an
equivalence relation $\sim_{(X,Y)}$ on $\Hom_\calV(X,Y)$. Then we
say that these equivalence relations form a compatible system if
they satisfy the following three conditions:
\begin{enumerate}
\item If $\alpha,\alpha':X\to Y$ are morphisms such that
$(\alpha,\alpha')\in S$ then $\alpha\sim_{(X,Y)}\alpha'$.
\item If $\alpha\sim_{(X,Y)}\alpha'$ and $\beta\sim_{(Y,Z)}\beta'$
then $\beta\circ\alpha\sim_{(X,Z)}\beta'\circ\alpha'$.
\item If $\alpha\sim_{(X,Y)}\alpha'$ and $\beta\sim_{(Z,W)}\beta'$ then
$\alpha\otimes\beta\sim_{(X\otimes Z,Y\otimes W)}\alpha'\otimes\beta'$.
\end{enumerate}
Let $\{\sim_{(X,Y)} : (X,Y)\in\objects(\calV)^2\}$ be the unique
minimal compatible system of equivalence relations. We define
$\calV/S$ to be the category with objects
$\objects(\calV/S):=\objects(\calV)$ and morphisms
$\Hom_{\calV/S}(X,Y):=\Hom_{\calV}(X,Y)/\sim_{(X,Y)}$. Then
$\calV/S$ is a Picard category in a natural way and there exists a
canonical monoidal functor $Q:\calV\to\calV/S$. Furthermore it is
easy to verify the universal property for $\calV/S$.

By construction, the functor $Q:\calV\to\calV/S$ is bijective on
objects and surjective on morphisms. Hence the induced homomorphisms
$\pi_i(Q): \pi_i(\calV)\to\pi_i(\calV/S)$ are certainly surjective.
The map $\pi_0(Q)$ is injective because any isomorphism in $\calV/S$
lifts to an isomorphism in $\calV$. Finally, if $\alpha, \alpha':
X\to Y$ are two distinct morphisms in $\calV$ with
$(\alpha,\alpha')\in S$, then the element
$\alpha^{-1}\circ\alpha'\in\Aut_\calV(X)=\pi_1(\calV)$ is
non-trivial but becomes trivial in $\pi_1(\calV/S)$.
\end{proof}

Recall that a ring $R$ is called regular if $R$ is noetherian and
every $R$-module has a finite projective resolution.

\begin{proposition}
\label{proposition_perf_regular} If the ring $R$ is regular, then
for every Picard category $\calP$ the functor
$\det\big(\Dperf(R),\calP\big) \to
\det\big((\Cperf(R),\qis),\calP\big)$ induced by
$I:\Cperf(R)\to\Dperf(R)$ is an equivalence of categories. Hence in
this case the homomorphism
\begin{equation*}
K_i\big(\Cperf(R),\qis\big) \longrightarrow K_i\big(\Dperf(R)\big)
\end{equation*}
is an isomorphism for $i=0$ and $i=1$.
\end{proposition}

\begin{proof}
Let $F: \det\big(\Dperf(R),\calP\big)\to\det(R\proj,\calP)$ be the
functor which is induced by the embedding $R\proj\to\Dperf(R)$. We
claim that $F$ is an equivalence of categories.

Let $R\modfg$ denote the abelian category of finitely generated
$R$-modules and $\Db(R\modfg)$ the bounded derived category of
$R\modfg$. Since $R$ is regular every bounded complex of finitely
generated $R$-modules is perfect. On the other hand, every perfect
complex of $R$-modules is isomorphic in $\Dperf(R)$ to a bounded
complex of finitely generated projective modules and so in
particular to a bounded complex of finitely generated modules. It
easily follows that there exists a canonical equivalence of
triangulated categories $\Db(R\modfg)\to\Dperf(R)$.

Now the functor $F: \det\big(\Dperf(R),\calP\big) \to
\det(R\proj,\calP)$ can be factored as
\begin{equation*}
\begin{split}
\det\big(\Dperf(R),\calP\big) & \stackrel{(1)}{\longrightarrow}
\det\big(\Db(R\modfg),\calP\big) \\
& \stackrel{(2)}{\longrightarrow} \det(R\modfg,\calP) \\
& \stackrel{(3)}{\longrightarrow} \det(R\proj,\calP).
\end{split}
\end{equation*}
Here the functor (1) is induced by the equivalence of triangulated
categories $\Db(R\modfg)\to\Dperf(R)$ and is therefore itself an
equivalence. The functor (2) is induced by the canonical functor
$R\modfg\to\Db(R\modfg)$ and is an equivalence by the theorem of the
heart \cite[Theorem 5.2]{Breuning06}. Finally the functor (3) is
induced by the inclusion $R\proj\to R\modfg$. By Quillen's
resolution theorem \cite[Corollary 2 in \S 4]{Quillen73} this
inclusion induces an isomorphism $K_i(R\proj)\cong K_i(R\modfg)$ for
all $i$, hence by Lemma \ref{lemma_equivalence} the functor (3) is
an equivalence. It follows that $F$ is an equivalence as claimed.

Now note that the equivalence $F$ can also be factored as
\begin{equation*}
\det\big(\Dperf(R),\calP\big)\longrightarrow
\det\big((\Cperf(R),\qis),\calP\big)\longrightarrow
\det(R\proj,\calP).
\end{equation*}
Since the functor
$\det\big((\Cperf(R),\qis),\calP\big)\to\det(R\proj,\calP)$ is an
equivalence by Corollary \ref{corollary_perfect_exact}, it follows
that $\det\big(\Dperf(R),\calP\big) \to
\det\big((\Cperf(R),\qis),\calP\big)$ is an equivalence. This proves
the first statement of the proposition.

We have shown that for every Picard category $\calP$ the functor
$\det\big(\Dperf(R),\calP\big) \to
\det\big((\Cperf(R),\qis),\calP\big)$ is an equivalence of
categories. By an argument similar to Lemma \ref{lemma_equivalence}
this implies that $K_i\big(\Cperf(R),\qis\big) \to
K_i\big(\Dperf(R)\big)$ is an isomorphism for $i=0$ and $i=1$.
\end{proof}

We can now prove Theorem \ref{theorem_intro} from the introduction.

\begin{proof}[Proof of Theorem \ref{theorem_intro}]
The composite of the canonical homomorphisms
\begin{equation*}
K_i(R)\longrightarrow K_i\big(\Cperf(R),\qis\big)\longrightarrow
K_i\big(\Dperf(R)\big)
\end{equation*}
from Corollary \ref{corollary_perfect_exact} and Lemma
\ref{lemma_C_D_perf_functor} gives a canonical homomorphism
$K_i(R)\to K_i\big(\Dperf(R)\big)$. The statements about bijectivity
and surjectivity follow from Corollary
\ref{corollary_perfect_exact}, Proposition
\ref{proposition_perf_universal} and Proposition
\ref{proposition_perf_regular}.
\end{proof}

\section{An example}
\label{section_example}

We have seen that for a regular ring $R$ the canonical map
$K_1(R)\to K_1\big(\Dperf(R)\big)$ is an isomorphism. In this
section we will give an example of a non-regular ring $R$ for which
the groups $K_1(R)$ and $K_1\big(\Dperf(R)\big)$ are not isomorphic.
The same example also shows that in general the groups $K_1(R\proj)$
and $K_1\big(\Db(R\proj)\big)$ are not isomorphic, so the
isomorphism $K_1(\calA)\cong K_1\big(\Db(\calA)\big)$ for an abelian
category $\calA$ (compare \cite[\S 5.1]{Breuning06}) does not
generalize to exact categories. The example in this section is
motivated by \cite[\S 2]{Vaknin01b}.

For any ring $R$ we let $\Kb(R\proj)$ be the bounded homotopy
category of $R\proj$, so the objects of $\Kb(R\proj)$ are bounded
complexes of finitely generated projective $R$-modules and the
morphisms are homotopy classes of morphisms of complexes. It is well
known that $\Kb(R\proj)$ has the structure of a triangulated
category.

There exists a canonical functor $\Cb(R\proj)\to\Kb(R\proj)$. A
quasi-isomorphism $a: A\to B$ in $\Cb(R\proj)$ is mapped to an
isomorphism in $\Kb(R\proj)$. Hence, as in the proof of Lemma
\ref{lemma_C_D_perf_functor}, if $\Delta: 0\to A\xrightarrow{a}
B\xrightarrow{b} C\to 0$ is a short exact sequence in $\Cb(R\proj)$,
there exists a canonical morphism $c: C\to TA$ in $\Kb(R\proj)$ such
that $\widehat{\Delta}: A\xrightarrow{a} B\xrightarrow{b}
C\xrightarrow{c} TA$ is a distinguished triangle in $\Kb(R\proj)$.

\begin{lemma}
\label{lemma_example} For any ring $R$ the canonical functor
$\Cb(R\proj)\to\Kb(R\proj)$ induces a functor
\begin{equation}
\label{equation_example_1} \det\big(\Kb(R\proj),\calP\big)
\longrightarrow \det\big((\Cb(R\proj),\qis),\calP\big)
\end{equation}
and therefore a homomorphism
\begin{equation}
\label{equation_example_2} K_1\big(\Cb(R\proj),\qis\big)
\longrightarrow K_1\big(\Kb(R\proj)\big).
\end{equation}
The homomorphism (\ref{equation_example_2}) is always surjective. If
$R=k[\varepsilon]/(\varepsilon^2)$ for a field $k$ then the
homomorphism (\ref{equation_example_2}) is not injective.
\end{lemma}

\begin{proof}
The proof of the existence of the functor (\ref{equation_example_1})
and homomorphism (\ref{equation_example_2}) is essentially the same
as the proof of Lemma \ref{lemma_C_D_perf_functor}, and the proof of
the surjectivity of (\ref{equation_example_2}) is similar to the
proof of Proposition \ref{proposition_perf_universal}. More
precisely, if $f=(f_1,f_2): (\Cb(R\proj),\qis)\to\calV$ is a
universal determinant functor then we can construct a universal
determinant functor $g=(g_1,g_2): \Kb(R\proj)\to\calW$ where $\calW$
is obtained from $\calV$ by identifying certain homomorphisms. We
denote the corresponding monoidal functor $\calV\to\calW$ by $Q$. It
follows from Lemma \ref{lemma_quotient} that
(\ref{equation_example_2}) is surjective.

From now on let $R=k[\varepsilon]/(\varepsilon^2)$ for some field
$k$. To show that the homomorphism (\ref{equation_example_2}) is not
injective, it suffices to show that there exist two morphisms
$\alpha, \alpha': X\to Y$ in $\calV$ such that $\alpha\neq\alpha'$
but $Q(\alpha)=Q(\alpha')$. We claim that the two morphisms
$\alpha=\id: f_1(R[0])\to f_1(R[0])$ and
$\alpha'=f_1(1+\varepsilon): f_1(R[0])\to f_1(R[0])$ (where the
homomorphism $1+\varepsilon: R[0]\to R[0]$ is given by
multiplication with $1+\varepsilon$) have these properties.

We first show that $\alpha\neq\alpha'$. Recall that $K_1(R)$ can be
described in terms of generators and relations, where the generators
are pairs $(P,a)$ with $P\in\objects(R\proj)$ and $a:P\to P$ an
automorphism. Since $R$ is a commutative local ring, the usual
determinant gives an isomorphism $K_1(R)\cong R^\times$. It follows
that $(R,\id)$ and $(R,1+\varepsilon)$ are distinct in $K_1(R)$.
Since $K_1(R)\cong K_1(R\proj)$ we can deduce that $h_1(\id)$ and
$h_1(1+\varepsilon)$ are distinct in $K_1(R\proj)$ where $(h_1,h_2)$
is a universal determinant functor on $R\proj$. Finally this implies
that $\alpha=f_1(\id)$ and $\alpha'=f_1(1+\varepsilon)$ are distinct
in $K_1\big(\Cb(R\proj),\qis\big)$ because the canonical map
$K_1(R\proj)\to K_1\big(\Cb(R\proj),\qis\big)$ is an isomorphism.

Next we show that $Q(\alpha)=Q(\alpha')$. For this it suffices to
show that $g_1(\id)=g_1(1+\varepsilon)$. However this follows
immediately from the commutative diagram in $\Kb(R\proj)$
\begin{equation*}
\xymatrix{ R[0] \ar[r]^{\varepsilon} \ar@{=}[d] & R[0] \ar[r]
\ar[d]^{1+\varepsilon} & C \ar[r] \ar@{=}[d] & R[1] \ar@{=}[d] \\
R[0] \ar[r]^{\varepsilon} & R[0] \ar[r] & C \ar[r] & R[1] }
\end{equation*}
where $C=\cone\big(R[0]\xrightarrow{\varepsilon} R[0]\big)$, compare
\cite[\S 2]{Vaknin01b}.
\end{proof}

\begin{corollary}
\label{corollary_example_0} If $k$ is a finite field and
$R=k[\varepsilon]/(\varepsilon^2)$ then the groups $K_1(R)$ and
$K_1\big(\Kb(R\proj)\big)$ are not isomorphic.
\end{corollary}

\begin{proof}
Lemma \ref{lemma_example} shows that in this case the canonical map
$K_1(R)\to K_1\big(\Kb(R\proj)\big)$ is surjective but not
injective. Since $K_1(R)\cong R^\times$ is finite, it follows that
$K_1(R)$ and $K_1\big(\Kb(R\proj)\big)$ are not isomorphic.
\end{proof}

\begin{corollary}
\label{corollary_example_1} If $k$ is a finite field and
$R=k[\varepsilon]/(\varepsilon^2)$ then the groups
$K_1(R)=K_1(R\proj)$ and $K_1\big(\Db(R\proj)\big)$ are not
isomorphic.
\end{corollary}

\begin{proof}
Recall that the derived category $\Db(R\proj)$ is obtained from
$\Kb(R\proj)$ by inverting all morphisms whose cone is acyclic (in
the sense of \cite[p.\ 389]{Neeman90}). But it is easy to see that
the cone of a morphism $a: A\to B$ in $\Kb(R\proj)$ is acyclic if
and only if $a$ is an isomorphism. Hence $\Db(R\proj)=\Kb(R\proj)$
and therefore $K_1\big(\Db(R\proj)\big)=K_1\big(\Kb(R\proj)\big)$.
Thus Corollary \ref{corollary_example_1} follows from Corollary
\ref{corollary_example_0}.
\end{proof}

\begin{corollary}
\label{corollary_example_2} If $k$ is a finite field and
$R=k[\varepsilon]/(\varepsilon^2)$ then the groups $K_1(R)$ and
$K_1\big(\Dperf(R)\big)$ are not isomorphic.
\end{corollary}

\begin{proof}
The canonical functor $\Kb(R\proj)\to\Dperf(R)$ is an equivalence of
triangulated categories. Hence $K_1\big(\Kb(R\proj)\big)\cong
K_1\big(\Dperf(R)\big)$ by \cite[Corollary 4.11]{Breuning06}. Thus
Corollary \ref{corollary_example_2} follows from Corollary
\ref{corollary_example_0}.
\end{proof}

\section{Homotopy fibres of monoidal functors}
\label{section_homotopy_fibre}

In this section we summarize the necessary facts about the homotopy
fibre of a monoidal functor of Picard categories. These
constructions and results are well-known (cf.\ \cite{BurnsFlach01},
\cite[\S 5]{BreuningBurns05}), the only difference in our
presentation here is the absence of a fixed unit object.

Let $M=(M,c): \calP\to\calP'$ be a monoidal functor of Picard
categories. The homotopy fibre of $M$ is the Picard category
$\calF(M)$ defined as follows. Objects of $\calF(M)$ are pairs
$(X,\delta)$ where $X$ is an object of $\calP$ and $\delta: M(X)\to
M(X)\otimes M(X)$ is a unit structure on $M(X)$. A morphism
$(X,\delta)\to (Y,\varepsilon)$ in $\calF(M)$ is a morphism $\alpha:
X\to Y$ in $\calP$ such that $\varepsilon\circ
M(\alpha)=(M(\alpha)\otimes M(\alpha))\circ\delta$. The composition
of morphisms in $\calF(M)$ is given by the composition of morphisms
in $\calP$.

The $\otimes$-product of $(X,\delta)$ and $(Y,\varepsilon)$ is
$(X,\delta)\otimes (Y,\varepsilon)=(X\otimes Y,\gamma)$ where
$\gamma$ is induced by the isomorphism $M(X\otimes
Y)\xrightarrow{c_{X,Y}} M(X)\otimes M(Y)$ and the product unit
structure $(M(X),\delta)\otimes (M(Y),\varepsilon)$, i.e.\ $\gamma$
is the composite isomorphism $M(X\otimes Y)\cong M(X)\otimes M(Y)
\xrightarrow{\delta\otimes\varepsilon} \big(M(X)\otimes
M(X)\big)\otimes\big(M(Y)\otimes M(Y)\big)\cong \big(M(X)\otimes
M(Y)\big)\otimes\big(M(X)\otimes M(Y)\big)\cong M(X\otimes Y)\otimes
M(X\otimes Y)$. The $\otimes$-product of two morphisms in $\calF(M)$
is simply the $\otimes$-product of these morphisms in $\calP$.

The AC-structure on $\calF(M)$ is induced by the AC-tensor structure
on $\calP$, i.e.\ $\psi_{(X,\delta),(Y,\varepsilon)}:
(X,\delta)\otimes (Y,\varepsilon)\cong (Y,\varepsilon)\otimes
(X,\delta)$ is given by $\psi_{X,Y}: X\otimes Y\cong Y\otimes X$,
and
$\varphi_{(X,\delta),(Y,\varepsilon),(Z,\gamma)}=\varphi_{X,Y,Z}$.

There exists an obvious monoidal functor $J: \calF(M)\to\calP$ which
sends an object $(X,\delta)$ to $X$ and a morphism $\alpha$ to
$\alpha$. Applying $\pi_i$ to this functor gives homomorphisms
$\pi_i(J): \pi_i(\calF(M))\to\pi_i(\calP)$ for $i=0,1$. Applying
$\pi_i$ to the functor $M:\calP\to\calP'$ gives homomorphisms
$\pi_i(M):\pi_i(\calP)\to\pi_i(\calP')$ for $i=0,1$. Finally there
is a homomorphism $\partial^1: \pi_1(\calP')\to\pi_0(\calF(M))$
which sends $\alpha\in\pi_1(\calP')$ to the isomorphism class of
$(U,\delta)$ where $\gamma: U\cong U\otimes U$ is any unit in
$\calP$ and $\delta$ is the composite isomorphism
$M(U)\xrightarrow{\alpha} M(U)\xrightarrow{M(\gamma)} M(U\otimes
U)\xrightarrow{c_{U,U}} M(U)\otimes M(U)$.

The following lemma is well known and easy to verify.

\begin{lemma}
There is an exact sequence of homotopy groups
\begin{equation*}
0 \to \pi_1(\calF(M))\xrightarrow{\pi_1(J)} \pi_1(\calP)
\xrightarrow{\pi_1(M)} \pi_1(\calP') \xrightarrow{\partial^1}
\pi_0(\calF(M))\xrightarrow{\pi_0(J)} \pi_0(\calP)
\xrightarrow{\pi_0(M)} \pi_0(\calP').
\end{equation*}
\end{lemma}

The homotopy fibre and the associated exact sequence of homotopy
groups are functorial in the following sense. Given a diagram
\begin{equation*}
\xymatrix{
\calP \ar[r]^{M} \ar[d]^{A} & \calP' \ar[d]^{B} \\
\calQ \ar[r]^{N} & \calQ' }
\end{equation*}
of Picard categories and monoidal functors, and an isomorphism
$\kappa: B\circ M\to N\circ A$ of monoidal functors, we obtain a
monoidal functor $\calF(M)\to\calF(N)$ of the homotopy fibres, which
sends an object $(X,\delta)$ in $\calF(M)$ to the object
$(A(X),\delta')$ in $\calF(N)$, where $\delta'$ is the composite
$N(A(X))\xrightarrow{\kappa_X^{-1}} B(M(X))\xrightarrow{B(\delta)}
B(M(X)\otimes M(X)) \cong B(M(X))\otimes
B(M(X))\xrightarrow{\kappa_X\otimes\kappa_X} N(A(X))\otimes
N(A(X))$, and a morphism $\alpha: (X,\delta)\to (Y,\varepsilon)$ to
the morphism $A(\alpha)$. It is not difficult to verify that the
induced homomorphisms $\pi_i(\calF(M))\to\pi_i(\calF(N))$ for
$i=0,1$ make the diagram
\begin{equation*}
\xymatrix@C-0.2cm{ 0 \ar[r] & \pi_1(\calF(M)) \ar[r] \ar[d] &
\pi_1(\calP) \ar[r] \ar[d]^{\pi_1(A)} & \pi_1(\calP') \ar[r]
\ar[d]^{\pi_1(B)} & \pi_0(\calF(M)) \ar[r] \ar[d] &
\pi_0(\calP) \ar[r] \ar[d]^{\pi_0(A)} & \pi_0(\calP') \ar[d]^{\pi_0(B)} \\
0 \ar[r] & \pi_1(\calF(N)) \ar[r] & \pi_1(\calQ) \ar[r] &
\pi_1(\calQ') \ar[r] & \pi_0(\calF(N)) \ar[r] & \pi_0(\calQ) \ar[r]
& \pi_0(\calQ') }
\end{equation*}
commutative.

\begin{remark}
Giving the structure of a unit on $M(X)$ is equivalent to giving an
isomorphism $M(X)\to 1_{\calP'}$ where $1_{\calP'}\to
1_{\calP'}\otimes 1_{\calP'}$ is a fixed unit in $\calP'$. Therefore
the homotopy fibre defined above agrees with the fibre product
considered in \cite{BurnsFlach01}.
\end{remark}

\section{Euler characteristics via triangulated categories}
\label{section_Euler_triangulated}

In this section we describe the construction of Euler
characteristics in a relative algebraic $K_0$-group. The
construction here is essentially the same as in
\cite{BreuningBurns05}, except that we work with determinant
functors on triangulated categories of perfect complexes instead of
determinant functors on exact categories of bounded complexes.

Let $R\to S$ be a homomorphism of rings such that $S$ is flat
as right $R$-module. Furthermore we assume that $S$ is regular.
Let $K_0(R,S)$ be the relative algebraic $K$-group which is defined
in terms of generators and relations in \cite[p.\ 215]{Swan68}.

We fix universal determinant functors $g_R: \Dperf(R)\to\calW(R)$ and
$g_S: \Dperf(S)\to\calW(S)$. If $F: \Dperf(R)\to\Dperf(S)$ denotes
the functor given by the scalar extension $P\mapsto S\otimes_R P$
then there exists a monoidal functor $N: \calW(R)\to\calW(S)$ such
that the determinant functors $N\circ g_R$ and $g_S\circ F$ are
isomorphic. We fix such a functor $N$ and isomorphism $\mu: N\circ
g_R\cong g_S\circ F$.

\begin{lemma}
\label{lemma_K_0_D_iso} Let $N: \calW(R)\to\calW(S)$ and $\mu:
N\circ g_R\cong g_S\circ F$ be as above, and let $\calF(N)$ denote
the homotopy fibre of $N$. Then there exists an isomorphism $\eta:
K_0(R,S)\cong \pi_0(\calF(N))$ (depending on $\mu$).
\end{lemma}

We will prove Lemma \ref{lemma_K_0_D_iso} below. The isomorphism
$\eta: K_0(R,S)\cong\pi_0(\calF(N))$ allows us to construct
invariants in $K_0(R,S)$ using the determinant functors $g_R$ and
$g_S$. If $C$ is a perfect complex of $R$-modules and $\varepsilon$
is a unit structure on $N(g_R(C))$, then $(g_R(C),\varepsilon)$ is
an object in $\calF(N)$ and therefore has a class in
$\pi_0(\calF(N))\cong K_0(R,S)$. The isomorphism $\mu_C:
N(g_R(C))\cong g_S(S\otimes_R C)$ and the fact that $S$ is regular
allow us to construct a unit structure on $N(g_R(C))$ from certain
information about the cohomology of $S\otimes_R C$. The relevant
properties of the determinant of the cohomology are summarized in
the following lemma (which is proved later in this section).

\begin{lemma}
\label{lemma_tri_cohom} Let $S$ be a regular ring and $g:
\Dperf(S)\to\calP$ a determinant functor. Let $P$ be a perfect
complex of $S$-modules.
\begin{enumerate}
\item Let $H(P)$ denote the cohomology of $P$ considered as a
complex with zero differentials. Then $H(P)$ is a bounded complex
of finitely generated $S$-modules
(and so in particular it lies in $\Dperf(S)$), and there
exists a canonical isomorphism
\begin{equation*}
g(P)\cong g(H(P)).
\end{equation*}
\item Let $H^{\ev}(P)$ resp.\ $H^{\od}(P)$ denote the direct sum of the
even resp.\ odd cohomology of $P$. Then there exists a canonical
isomorphism
\begin{equation*}
g(H(P))\cong g(H^{\ev}(P)[0])\otimes g(H^{\od}(P)[1]).
\end{equation*}
\item There exists a canonical unit structure on
$g(H^{\od}(P)[0])\otimes g(H^{\od}(P)[1])$.
\end{enumerate}
\end{lemma}

Using Lemmas \ref{lemma_K_0_D_iso} and \ref{lemma_tri_cohom} we can now define
Euler characteristics in $K_0(R,S)$. To simplify the notation we will
write $C_S$ for $S\otimes_R C$ if $C$ is a complex of $R$-modules.

\begin{definition}
Let $C$ be a perfect complex of $R$-modules and
$t: H^{\ev}(C_S)\xrightarrow{\cong} H^{\od}(C_S)$ an isomorphism of
$S$-modules. We define a unit structure $\varepsilon$ on $N(g_R(C))$
via the composite isomorphism
\begin{equation*}
\begin{split}
N(g_R(C)) & \xrightarrow{\mu_C} g_S(C_S) \\
& \xrightarrow{\cong} g_S(H(C_S)) \\
& \xrightarrow{\cong} g_S(H^{\ev}(C_S)[0])\otimes g_S(H^{\od}(C_S)[1]) \\
& \xrightarrow{g_S(t)\otimes\id} g_S(H^{\od}(C_S)[0])\otimes
g_S(H^{\od}(C_S)[1])
\end{split}
\end{equation*}
and the canonical unit structure on
$g_S(H^{\od}(C_S)[0])\otimes g_S(H^{\od}(C_S)[1])$. Then the Euler
characteristic $\chi^{\tri}(C,t)\in K_0(R,S)$ of the pair
$(C,t)$ is defined to be the isomorphism class of the object
$(g_R(C),\varepsilon)$ in $\pi_0(\calF(N))\xrightarrow{\eta^{-1}}
K_0(R,S)$.
\end{definition}

\begin{lemma}
\label{lemma_Euler_D_indep} The definition of $\chi^{\tri}(C,t)\in
K_0(R,S)$ is independent of all choices.
\end{lemma}

\begin{lemma}
\label{lemma_agreement} Let $\chi(C,t)$ be the Euler characteristic
defined in \cite[Definition 5.5]{BreuningBurns05}. Then
$\chi^{\tri}(C,t)=\chi(C,t)$.
\end{lemma}

We remark that Lemma \ref{lemma_Euler_D_indep} follows immediately from
Lemma \ref{lemma_agreement} and the independence of $\chi(C,t)$ of all
choices. However the latter independence was only stated without
proof in \cite[Remark 5.3]{BreuningBurns05}, and we will therefore include a
complete proof of Lemma \ref{lemma_Euler_D_indep} below.

\begin{proof}[Proof of Lemma \ref{lemma_K_0_D_iso}]
This proof is similar to \cite[Lemma 5.1]{BreuningBurns05}. We
define a homomorphism $\eta: K_0(R,S)\to \pi_0(\calF(N))$ by sending
a generator $(P,a,Q)$ to the isomorphism class of $(g_R(P[0])\otimes
g_R(Q[0])^{-1},\delta)$, where the unit structure $\delta$ on
$N(g_R(P[0])\otimes g_R(Q[0])^{-1})$ is obtained via the composite
isomorphism
\begin{equation*}
\begin{split}
N(g_R(P[0])\otimes g_R(Q[0])^{-1}) & \xrightarrow{\cong}
N(g_R(P[0]))\otimes N(g_R(Q[0]))^{-1} \\
& \xrightarrow{\mu} g_S(S\otimes_R P[0])\otimes g_S(S\otimes_R
Q[0])^{-1} \\
& \xrightarrow{g_S(a[0])\otimes\id} g_S(S\otimes_R Q[0])\otimes
g_S(S\otimes_R Q[0])^{-1}
\end{split}
\end{equation*}
(where $\mu$ denotes the isomorphism induced by $\mu_{P[0]}: N(g_R(P[0]))\cong g_S(S\otimes_R P[0])$
and $\mu_{Q[0]}: N(g_R(Q[0]))\cong g_S(S\otimes_R Q[0])$)
and the canonical unit structure on $g_S(S\otimes_R Q[0])\otimes
g_S(S\otimes_R Q[0])^{-1}$. We now have a commutative diagram with
exact rows
\begin{equation*}
\xymatrix@C-0.3cm{ K_1(R) \ar[r] \ar[d] & K_1(S) \ar[r] \ar[d] &
K_0(R,S)
\ar[r] \ar[d]^{\eta} & K_0(R) \ar[r] \ar[d] & K_0(S) \ar[d] \\
K_1(\Dperf(R)) \ar[r] & K_1(\Dperf(S)) \ar[r] & \pi_0(\calF(N))
\ar[r] & K_0(\Dperf(R)) \ar[r] & K_0(\Dperf(S)) }
\end{equation*}
in which the unlabeled vertical maps are the canonical homomorphisms
from Theorem \ref{theorem_intro}. Since the first vertical map is
surjective and the second, fourth and fifth vertical maps are
isomorphisms, the 5-Lemma implies that $\eta$ is an isomorphism.
\end{proof}

\begin{proof}[Proof of Lemma \ref{lemma_tri_cohom}]
The canonical functor $\Cb(S\modfg)\to\Dperf(S)$ induces a functor
\begin{equation*}
\det\big(\Dperf(S),\calP\big) \longrightarrow
\det\big((\Cb(S\modfg),\qis),\calP\big)
\end{equation*}
(this is the composite of the functor $\det\big(\Dperf(S),\calP\big)
\to \det\big((\Cperf(S),\qis),\calP\big)$ from Lemma
\ref{lemma_C_D_perf_functor} and the functor
$\det\big((\Cperf(S),\qis),\calP\big) \to
\det\big((\Cb(S\modfg),\qis),\calP\big)$ induced by
$\Cb(S\modfg)\to\Cperf(S)$). Let $f: (\Cb(S\modfg),\qis)\to\calP$
denote the image of $g: \Dperf(S)\to\calP$ under this functor.
\begin{enumerate}
\item Choose a complex $U$ in $\Cb(S\modfg)$ together with a
quasi-isomorphism $a: U\to P$. Then we get an isomorphism $H(a):
H(U)\to H(P)$, and it is clear that $H(U)$ (and hence $H(P)$) lies
in $\Cb(S\modfg)$. Now
\cite[Proposition 3.1]{BreuningBurns05} shows that there exists a
canonical isomorphism $f(U)\xrightarrow{\cong} f(H(U))$. Hence we
obtain an isomorphism $g(P)\xrightarrow{\cong} g(H(P))$ as the composite
\begin{equation*}
g(P)\xrightarrow{g(a)^{-1}}
g(U)=f(U)\xrightarrow{\cong} f(H(U))=g(H(U))\xrightarrow{g(H(a))} g(H(P)).
\end{equation*}
One easily checks that this composite isomorphism is independent of the
choice of $U$ and $a$.
\item This follows from the canonical isomorphism
$f(H(P))\cong f(H^{\ev}(P)[0])\otimes f(H^{\od}(P)[1])$,
see \cite[Proposition 4.4]{BreuningBurns05}.
\item This follows from the canonical unit structure on
$f(H^{\od}(P)[0])\otimes f(H^{\od}(P)[1])$,
see \cite[Lemma 2.3]{BreuningBurns05}.\qedhere
\end{enumerate}
\end{proof}

\begin{proof}[Proof of Lemma \ref{lemma_Euler_D_indep}]
Suppose that $g'_R: \Dperf(R)\to\calW'(R)$ and $g'_S:
\Dperf(S)\to\calW'(S)$ are also universal determinant functor, $N':
\calW'(R)\to\calW'(S)$ is a monoidal functor, and $\mu': N'\circ
g'_R\cong g'_S\circ F$ is an isomorphism. Then we can choose a
monoidal functor $A: \calW(R)\to\calW'(R)$ together with an
isomorphism $A\circ g_R\cong g'_R$, and a monoidal functor $B:
\calW(S)\to\calW'(S)$ together with an isomorphism $B\circ g_S\cong
g'_S$. So we have a diagram of triangulated/Picard categories and
exact/determinant/monoidal functors as follows.
\begin{equation*}
\xymatrix{ \Dperf(R) \ar[rr]^{F} \ar[dr]^{g_R} \ar[dd]_{g'_R} & &
\Dperf(S) \ar[dr]^{g_S} \ar'[d][dd]_(0.4){g'_S} & \\
& \calW(R) \ar[rr]^(0.3){N} \ar[dl]^{A} & & \calW(S) \ar[dl]^{B} \\
\calW'(R) \ar[rr]^{N'} & & \calW'(S) & }
\end{equation*}

From $\mu$, $\mu'$ and the isomorphism of the triangles we obtain an
isomorphism of determinant functors
\begin{equation*}
\begin{split}
\Dperf(R) & \xrightarrow{g_R} \calW(R) \xrightarrow{N} \calW(S)
\xrightarrow{B} \calW'(S), \\
\Dperf(R) & \xrightarrow{g_R} \calW(R) \xrightarrow{A} \calW'(R)
\xrightarrow{N'} \calW'(S),
\end{split}
\end{equation*}
and since $g_R$ is universal
this isomorphism of determinant functors comes from
an isomorphism of monoidal functors $B\circ N\cong N'\circ
A$. Therefore by the results in \S \ref{section_homotopy_fibre} we
get an induced functor $Z:\calF(N)\to\calF(N')$ and thus a
homomorphism $\pi_0(Z): \pi_0(\calF(N))\to\pi_0(\calF(N'))$. One
easily sees that $\pi_0(Z)$ is in fact an isomorphism.

We now write $\eta_{N,\mu}: K_0(R,S)\to\pi_0(\calF(N))$ for the
isomorphism from Lemma \ref{lemma_K_0_D_iso} associated to $N$
and $\mu$, and $\eta_{N',\mu'}: K_0(R,S)\to\pi_0(\calF(N'))$ for the
isomorphism associated to $N'$ and $\mu'$. It is not difficult to
verify that
\begin{equation}
\label{equation_euler_tri_indep}
\eta_{N',\mu'}=\pi_0(Z)\circ\eta_{N,\mu}.
\end{equation}

On the other hand, let $(g_R(C),\varepsilon)$ be the object in
$\calF(N)$ which occurs in the construction of the Euler
characteristic with respect to $N$ and $\mu$, and let
$(g'_R(C),\varepsilon')$ be the object in $\calF(N')$ constructed
with respect to $N'$ and $\mu'$. It is then easy to see that
$Z((g_R(C),\varepsilon))\cong (g'_R(C),\varepsilon')$. Hence
$\pi_0(Z)$ sends the isomorphism class of $(g_R(C),\varepsilon)$ to
the isomorphism class of $(g'_R(C),\varepsilon')$. Together with
(\ref{equation_euler_tri_indep}) this implies the independence of
$\chi^{\tri}(C,t)$.
\end{proof}

\begin{proof}[Proof of Lemma \ref{lemma_agreement}]
It is easy to check that if $a: U\to C$ is a quasi-isomorphism and
$t': H^{\ev}(U_S)\to H^{\od}(U_S)$ is the composite isomorphism
\begin{equation*}
H^{\ev}(U_S)\xrightarrow{H^{\ev}(a_S)} H^{\ev}(C_S)\xrightarrow{t}
H^{\od}(C_S)\xrightarrow{H^{\od}(a_S)^{-1}} H^{\od}(U_S),
\end{equation*}
then $\chi^{\tri}(C,t)=\chi^{\tri}(U,t')$ and
$\chi(C,t)=\chi(U,t')$. Therefore we can assume from now on that
$C$ is a bounded complex of finitely generated projective $R$-modules.

Let $f_R: (\Cb(R\proj),\qis) \to \calV(R)$ and $f_S:
(\Cb(S\modfg),\qis) \to \calV(S)$ be universal determinant functors.
Let $E: \Cb(R\proj)\to\Cb(S\modfg)$ be the exact functor given by
$P\mapsto S\otimes_R P$, and let $M: \calV(R)\to\calV(S)$ be a
monoidal functor such that there exists an isomorphism of
determinant functors $\lambda: M\circ f_R\cong f_S\circ E$. Then by
\cite[Lemma 5.1]{BreuningBurns05} there exists an isomorphism $\xi:
K_0(R,S)\cong\pi_0(\calF(M))$ (depending on $\lambda$).

We write $I_R$ for the canonical functor $\Cb(R\proj)\to\Dperf(R)$.
Then there exists a monoidal functor $A: \calV(R)\to\calW(R)$ such
that the determinant functors $g_R\circ I_R$ and $A\circ f_R$ are
isomorphic. We fix such a functor $A$ and isomorphism of determinant
functors. Similarly, if $I_S:\Cb(S\modfg)\to\Dperf(S)$ is the
canonical functor, we can fix a monoidal functor $B:
\calV(S)\to\calW(S)$ and isomorphism $g_S\circ I_S\cong B\circ f_S$.

We now have the following diagram of exact/triangulated/Picard
categories and exact/determinant/monoidal functors.
\begin{equation*}
\xymatrix{ (\Cb(R\proj),\qis) \ar[rr]^{E} \ar[dd]^(0.7){I_R}
\ar[dr]^{f_R} & &
(\Cb(S\modfg),\qis) \ar'[d][dd]^(0.4){I_S} \ar[dr]^{f_S} & \\
& \calV(R) \ar[rr]^(0.35){M} \ar[dd]^(0.7){A} & &
\calV(S) \ar[dd]^(0.7){B} \\
\Dperf(R) \ar'[r][rr]^(0.35){F} \ar[dr]^{g_R} & &
\Dperf(S) \ar[dr]^{g_S} & \\
& \calW(R) \ar[rr]^{N} & & \calW(S) }
\end{equation*}
The square at the back is commutative. The squares on the left,
right, top and bottom are commutative up to fixed isomorphisms of
determinant functors. Hence we obtain an isomorphism of determinant
functors
\begin{gather*}
(\Cb(R\proj),\qis) \xrightarrow{f_R} \calV(R)
\xrightarrow{M} \calV(S) \xrightarrow{B} \calW(S), \\
(\Cb(R\proj),\qis) \xrightarrow{f_R} \calV(R) \xrightarrow{A} \calW(R)
\xrightarrow{N} \calW(S),
\end{gather*}
and thus (since $f_R$ is universal) also an isomorphism of monoidal
functors $B\circ M\cong N\circ A$. As shown in \S
\ref{section_homotopy_fibre} this implies that there is a monoidal
functor $Z: \calF(M)\to\calF(N)$. It follows easily from the proofs
of Lemma \ref{lemma_K_0_D_iso} and \cite[Lemma 5.1]{BreuningBurns05}
that
\begin{equation}
\label{equation_agree_1}
\eta=\pi_0(Z)\circ\xi.
\end{equation}

Now recall that the Euler characteristic $\chi(C,t)\in K_0(R,S)$ is defined
to be the isomorphism class of $(f_R(C),\delta)$ in
$\pi_0(\calF(M))\xrightarrow{\xi^{-1}} K_0(R,S)$, where $\delta$ is the unit
structure on $M(f_R(C))$ coming from the isomorphism
\begin{equation*}
\begin{split}
M(f_R(C)) \xrightarrow{\lambda_C} f_S(C_S) & \xrightarrow{\cong}
f_S(H(C_S))
\xrightarrow{\cong} f_S(H^{\ev}(C_S)[0])\otimes f_S(H^{\od}(C_S)[1]) \\
& \xrightarrow{f_S(t)\otimes\id} f_S(H^{\od}(C_S)[0])\otimes f_S(H^{\od}(C_S)[1])
\end{split}
\end{equation*}
and the canonical unit structure on
$f_S(H^{\od}(C_S)[0])\otimes f_S(H^{\od}(C_S)[1])$.

It is straightforward to check that
$Z((f_R(C),\delta))\cong (g_R(C),\varepsilon)$ where
$(g_R(C),\varepsilon)$ is the object in $\calF(N)$ which is used in the
definition of $\chi^{\tri}(C,t)$. Because of (\ref{equation_agree_1})
this implies that $\chi^{\tri}(C,t)=\chi(C,t)$
as required.
\end{proof}


\end{document}